\documentclass[12pt]{amsart}

\usepackage{amsmath}
\usepackage{amssymb}
\usepackage{amsthm}
\usepackage{mathrsfs}
\usepackage{comment}
\usepackage{hyperref}

\usepackage[all,cmtip]{xy}\usepackage{xcolor}
\usepackage{enumerate}
\usepackage{bm}
\usepackage{dsfont}
\usepackage{mathtools}
\usepackage[cal=euler]{mathalfa}
\usepackage[top=1in, bottom=1.25in, left=1.25in, right=1.25in]{geometry}
\usepackage{parskip}

\hypersetup{colorlinks=true,linkcolor=magenta,citecolor=blue}

\DeclareMathAlphabet{\mathpzc}{OT1}{pzc}{m}{it}

\theoremstyle{plain}
\newtheorem{theorem}{Theorem}[section]
\newtheorem*{theorem*}{Theorem}
\newtheorem{maintheorem}{Theorem}
\newtheorem{lemma}[theorem]{Lemma}
\newtheorem{proposition}[theorem]{Proposition}

\newtheorem{corollary}[theorem]{Corollary}

\theoremstyle{definition}
\newtheorem{definition}[theorem]{Definition}

\newtheorem{hypothesis}[theorem]{Hypothesis}

\newtheorem{remarks}[theorem]{Remarks}
\newtheorem{remark}[theorem]{Remark}

\numberwithin{equation}{section}
\numberwithin{figure}{section}

\newcommand{\Res}{\mathrm{Res}} %restriction
\newcommand{\Ind}{\mathrm{Ind}} %induction
\newcommand{\cInd}{\textrm{c-}\mathrm{Ind}}
\newcommand{\Hom}{\mathrm{Hom}}
\newcommand{\Spec}{\mathrm{Spec}}
\newcommand{\Rep}{\mathrm{Rep}}

\newcommand{\op}{\mathrm{op}}

\newcommand{\Gal}{\mathrm{Gal}}

\renewcommand{\sc}{\mathrm{sc}}
\newcommand{\Cusp}{\mathrm{Cusp}}

\newcommand{\Reg}{\mathrm{Reg}}
\newcommand{\Typ}{\mathrm{Typ}}

\newcommand{\bG}{\mathbf{G}}

\newcommand{\bH}{\mathbf{H}}

\newcommand{\apart}{\mathscr{A}}

\newcommand{\buil}{\mathscr{B}}
\newcommand{\C}{\mathbb{C}}

\newcommand{\red}{\mathrm{red}}

\newcommand{\rH}{\mathrm{H}}
\newcommand{\rec}{\mathrm{rec}}
\newcommand{\iner}{\mathrm{iner}}
\newcommand{\Type}{\mathrm{Type}}

\newcommand{\Irr}{\mathrm{Irr}}

\newcommand{\aff}{\mathrm{aff}}

\begin{document}

\setlength{\parindent}{0pt}
\title[Typical representations and inertial Langlands]{Typical representations, parabolic induction and the inertial local Langlands correspondence}
\author{Peter Latham}
\address{Department of Mathematics and Statistics, University of Ottawa, Ottawa, Canada}
\email{platham@uottawa.ca}
\date{\today}

\begin{abstract}
\noindent We prove a result which provides a link between the decomposition of parabolically induced representations and the Bushnell--Kutzko theory of typical representations. As an application, we show that there exists a well-defined inertial Langlands correspondence which respects the monodromy action of $L$-parameters, under some standard conjectures regarding the local Langlands correspondence. To allow for potential applications of this inertial Langlands correspondence, we also provide a complete construction of the set of typical representations, giving a parametrization of these in terms of the structure of the Bruhat--Tits building of $G$.
\end{abstract}
\maketitle
%%% \setlength{\parskip}{0pt}
%%% \tableofcontents
\setlength{\parskip}{12pt}
\setlength{\parindent}{0pt}

\section{Introduction}

In this paper, we show that there exists a well-defined ``inertial local Langlands correspondence'', as long as one assumes some standard conjectures about the existence and properties of the local Langlands correspondence. Recall that the local Langlands correspondence is a conjectural parametrization of the irreducible representations of a $p$-adic group $G$ in terms of \emph{$L$-parameters}, which are objects closely related to continuous Galois representations (see Section \ref{sec:l-params} for a precise definition).

The inertial Langlands correspondence is a modification of the local Langlands correspondence which tracks only the information which occurs ``on the compact level''---it replaces $L$-parameters with the restrictions to the compact inertia subgroup, and replaces irreducible\footnote{In the main body of the paper, we consider only \emph{regular} representations due to non-regular representations introducing additional technical complications. We overlook this distinction during the introduction.} representations with (objects closely related to) their \emph{types}, in the sense of \cite{BushnellKutzko1998}.

The existence of the inertial Langlands correspondence not only provides an affirmative answer to the natural question of whether the many available explicit constructions translate over the local Langlands correspondence, but is also of arithmetic significance. Most notably, the inertial Langlands correspondence is a key pre-requisite to the formation of the Breuil--M\'{e}zard conjecture \cite{BreuilMezard2002,EmertonGee2014}, which is a central problem in modern algebraic number theory, the partial solution of which has already found numerous applications---for example, Kisin's proof of (many cases of) the 2-dimensional Fontaine--Mazur conjecture \cite{Kisin2009} is largely based on establishing the corresponding cases of the Breuil--M\'{e}zard conjecture.

Each of these applications relies on the inertial Langlands correspondence for $\mathbf{GL}_N(F)$, the existence and fine properties of which have been understood for quite some time. However, if one wishes to investigate similar questions within the more general context of the Langlands programme for an arbitrary $p$-adic group, then one requires a similar understanding of the inertial Langlands correspondence for such groups. This is something of a problematic topic within the literature---the term ``inertial Langlands correspondence'' actually refers to \emph{two} (closely related) folklore conjectures. The author is not aware of an exposition of either anywhere in the literature, and in fact one of the two is not provably well-defined, given the current state of the art.

To explain this distinction, recall that an $L$-parameter for $G$ consists of a continuous representation of the Weil--Deligne group $W_F\ltimes\C$, taking image in the Langlands dual group $^LG$. The action of $\C$ contained within such a representation is the \emph{monodromy action} of the $L$-parameter; the role of the monodromy action is to distinguish between distinct irreducible representations with equivalent supercuspidal supports.
\begin{itemize}
\item The most basic form of the inertial Langlands correspondence is a canonical map which assigns to each Bushnell--Kutzko type a map $I_F\rightarrow{}^LG$, where $I_F\subset W_F$ is the inertia group. The existence of this map is not a problem; given the existence of the local Langlands correspondence, together with a construction of types, deducing the existence of this inertial correspondence is an exercise in the definitions. The issue here is that one forgets the monodromy action, which leads to a somewhat coarse correspondence---for applications such as the Breul--M\'{e}zard conjecture, one must keep track of the monodromy action. For this reason, we refer to this construction as the \emph{weak inertial Langlands correspondence}.
\item One then expects the existence of the ``strong'' inertial Langlands correspondence. This is entirely conjectural aside from the cases of $\mathbf{GL}_N(F)$ and $\mathbf{SL}_N(F)$, with even these being poorly represented in the literature\footnote{At the end of Section \ref{sec:inertial}, we provide an extended discussion of exactly what is currently known in these cases, and to what extent it is available in published form.}. It predicts that by replacing Bushnell--Kutzko types with a well-chosen set of \emph{typical representations} (in the sense of \cite{BushnellKutzko1998}), one may refine the weak correspondence to a natural map which associates to each typical representation a representation $I_F\rightarrow{}^LG$, together with a monodromy action. This is something of a widely-held folklore belief within the community, but cannot be clearly stated given our current understanding of parabolically induced representations---there is no hope of precisely defining \emph{which} typical representations one should consider.
\end{itemize}

In this paper, we present what appears to be an extremely satisfactory remedy to the situation. We make a minor concession, in allowing the domain of the inertial Langlands correspondence to consist not of typical representations, but of certain \emph{explicitly defined} finite sets of typical representations. In doing so, one obtains a canonical map $\varphi$ which associates to each irreducible representation $\pi$ of $G$ the (finite-modulo-conjugacy) set of typical representations which are contained in $\pi|_K$, as $K$ ranges over the maximal compact subgroups of $G$. Call a set $X$ of conjugacy classes of typical representations \emph{admissible} if it lies in the image of $\varphi$. Our main result is the following (see Theorem \ref{thm:inertial} for a precise statement).

\begin{maintheorem}
Let $X$ be an admissible set of typical representations. For each irreducible representation $\pi$ of $G$ such that $\varphi(\pi)=X$, let $\iner(X,\pi)$ denote the restriction to $I_F\ltimes\C$ of the $L$-parameters corresponding to $\pi$. Then $\iner(X)=\iner(X,\pi)$ is independent of the choice of $\pi$. In particular, one obtains a canonical map $X\mapsto\iner(X)$ which associates to each admissible set $X$ of typical representations the restriction to $I_F\ltimes\C$ of an $L$-parameter.
\end{maintheorem}

Our proof takes place entirely on the $p$-adic side of the correspondence, with the difficulty lying in showing that the map $\iner$ is well-defined. This requires us to prove the following result, which may be of interest independent of any applications to the Langlands programme.

\begin{maintheorem}
Suppose that $\pi_1$ and $\pi_2$ are two irreducible representations of $G$ which have the same supercuspidal supports, and suppose furthermore that $\varphi(\pi_1)=\varphi(\pi_2)$. Then $\pi_1\simeq\pi_2$.
\end{maintheorem}

This is intriguing in that, while not allowing for a complete description of the Jordan--H\"{o}lder series of a parabolically induced representation, it does allow for a combinatorial description of the irreducible components of such a representation, in terms of sets of typical representations.

In order to apply either this description of (semisimplified) parabolically induced representations, or to make a useful application of the above inertial local Langlands correspondence, one requires a workable description of the various typical representations. Our final result provides just this, showing that the isomorphism classes of typical representations admit a combinatorial-geometric description via Bruhat--Tits theory. Specifically, for each inertial equivalence class $\frak{s}$ and each maximal compact subgroup $K$ of $G$, the $\frak{s}$-typical representations of $K$ occur as components of $\Ind_J^K\ \lambda$, where $(J,\lambda)$ is an $\frak{s}$-type. In Section \ref{sec:construction}, we construct an injective system of subrepresentations $V_i$ of $\Ind_J^K\ \lambda$ such that $\Ind_J^K\ \lambda=\varinjlim V_i$, and show that the maximal elements of $\{V_i\}$ (which are largely determined by a choice of vertex in a certain sub-building of the Bruhat--Tits building of $G$) are exactly the irreducible components of $\Ind_J^K\ \lambda$.

We now summarize the layout of the paper. Section \ref{sec:notation} briefly introduces the key notation which will be used throughout the paper. Section \ref{sec:types} recalls some key definitions from \cite{BushnellKutzko1998}, and Section \ref{sec:kimyu} sketches the construction of types due to Kim--Yu \cite{KimYu2017}. From Section \ref{sec:l-params} onwards, we turn towards the Langlands correspondence, introducing the notion of $L$-parameters, and setting up some useful functors. Section \ref{sec:kaletha} recalls Kaletha's construction \cite{Kaletha2019} of the local Langlands correspondence for regular supercuspidal representations; Section \ref{sec:parabolic} then extends this to regular non-cuspidal representations, under some standard conjectures. Section \ref{sec:typical} extends the notions of Section \ref{sec:types} in order to introduce typical representations, which are then used in Section \ref{sec:inertial} to construct our inertial Langlands correspondence. Finally, Section \ref{sec:construction} provides an explicit description of the set of typical representations.

\textbf{Acknowledgements.} This work was partially supported by an NSERC grant via the University of Ottawa. I thank Monica Nevins for asking the questions which eventually led to these results, and Maarten Solleveld for a number of helpful comments on an earlier version of the paper.

\section{Notation}\label{sec:notation}

Let $F$ be a locally compact field, complete with respect to a non-archimedean discrete valuation, and with finite residue field $\frak{f}$ of odd characteristic $p$. We fix, once and for all, a separable algebraic closure $\bar{F}/F$.

Let $\bG$ be a quasi-split connected reductive algebraic group defined over $F$, and suppose that there exists a finite tamely ramified extension $E/F$ such that $\bG\times_{\Spec\ F}\Spec\ E$ is split. Moreover, we assume that $p$ is coprime to the order of the Weyl group of $\bG\times_{\Spec\ F}\Spec\ E$\footnote{In particular, this hypothesis ensures that all representations and $L$-parameters are \emph{essentially tame} (see \cite{Fintzen2018}).}, as well as to the orders of the algebraic fundamental groups of both $\bG$ and its reductive dual. We set $G=\bG(F)$.

Our approach will be grounded in Bruhat--Tits theory. We write $\buil(G)$ for the Bruhat--Tits building of $G$; it should be understood that this means the \emph{enlarged} Bruhat--Tits building. On the rare occasion where we wish to work with the reduced building $\buil^\red(G)$, we will make this clear. In particular, we denote the canonical projection $\buil(G)\rightarrow\buil^\red(G)$ by $x\mapsto[x]$, and always enclose points in square brackets when we wish to consider them in the reduced building.

Given a point $x\in\buil(G)$, we write $G_x$ for the stabilizer of $G$, and $G_{x,0}$ for the parahoric subgroup contained in $G_x$ with finite index. The Moy--Prasad filtration (see \cite{MoyPrasad1994}) of $G_{x,0}$ will be denoted by $G_{x,r}$, $r\geq 0$, with the convention that $G_{x,r+}$ denotes $\bigcup_{s>r}G_{x,s}$. We write $G_{x,r:s}$ for the quotient $G_{x,s}/G_{x,s}$ whenever $s>r$; this quotient is abelian if $r\geq s/2$.

We denote smooth induction and compact induction by $\Ind_A^B\ -$ and $\cInd_A^B\ -$, respectively. In the cases where these functors coincide, we prefer to denote the functor by $\Ind_A^B\ -$. Moreover, in the case where $B=G$ and $A=P$ is a parabolic subgroup of $G$, then $\Ind_P^G\ -=\cInd_P^G\ -$; in this special case we commit a slight abuse of notation and use $\Ind_P^G\ -$ to denote \emph{normalized} parabolic induction, i.e. we pre-compose with the functor $-\otimes\delta_P^{-1/2}$, where $\delta_P$ is the modular function of $P$.

\section{Bushnell--Kutzko types}\label{sec:types}

In this section, we briefly recap the key definitions which are required in order to work with Bushnell--Kutzko types.

Given an irreducible representation $\pi$ of $G$, one may always find a Levi subgroup $M$ of $G$, together with a supercuspidal representation $\pi_M$ of $M$, such that $\pi$ is isomorphic to a subquotient of $\Ind_P^G\ \pi_M$, for some parabolic subgroup $P$ of $G$ with Levi factor $M$. The pair $(M,\pi_M)$ is uniquely determined up to $G$-conjugacy; one therefore obtains a \emph{supercuspidal support map} $\pi\mapsto\sc(\pi)$, which associates to $\pi$ the $G$-conjugacy class of $(M,\pi_M)$.

Given two pairs $(M,\rho)$ and $(M',\rho')$ consisting of supercuspidal representations of Levi subgroups of $G$, we say that $(M,\rho)$ is \emph{$G$-inertially equivalent} to $(M',\rho')$ if there exists a $g\in G$ and an unramified character $\omega$ of $M'$ such that $^gM=M'$ and $^g\rho\simeq\rho'\otimes\omega$. We write $[M,\rho]_G$ for the $G$-inertial equivalence class of $(M,\rho)$, and $\frak{B}(G)$ for the set of equivalence classes of such pairs. There is an \emph{inertial support} map $\frak{I}:\Irr(G)\rightarrow\frak{B}(G)$ which maps $\pi$ to the inertial equivalence class of $\sc(\pi)$. 

For each class $\frak{s}\in\frak{B}(G)$, consider the full subcategory $\Rep^{\frak{s}}(G)$ consisting of those representations every irreducible subquotient of which has inertial support $\frak{s}$.

\begin{theorem}[\cite{Bernstein1984}]
Each category $\Rep^{\frak{s}}(G)$ is indecomposable, and one has a block decomposition
\[\Rep(G)=\prod_{\frak{s}\in\frak{B}(G)}\Rep^{\frak{s}}(G).
\]
\end{theorem}

\begin{definition}[\cite{BushnellKutzko1998}]
Let $(J,\lambda)$ be a pair consisting of an irreducible representation $\lambda$ of a compact open subgroup $J$ of $G$. Consider the full subcategory $\Rep_\lambda(G)$ of $\Rep(G)$ consisting of those representations which are generated by their $\lambda$-isotypic subspaces.

We say that $(J,\lambda)$ is an $\frak{s}$\emph{-type} if $\Rep_\lambda(G)=\Rep^{\frak{s}}(G)$.
\end{definition}

Next, we recall the notion of $G$-covers from \cite{BushnellKutzko1998}, which will be crucial to our approach.

Fix a Levi subgroup $M\subset G$ and a supercuspidal representation $\pi_M$ of $M$; write $\frak{s}_M=[M,\pi_M]_M$, and suppose that $(J_M,\lambda_M)$ is an $\frak{s}_M$-type. Let $\frak{s}=[M,\pi_M]_G$.

\begin{definition}
A \emph{$G$-cover} of $(J_M,\lambda_M)$ is a pair $(J,\lambda)$ consisting of an irreducible representation $\lambda$ of a compact open subgroup $J$ of $G$, satisfying the following properties.
\begin{enumerate}[(i)]
\item $J\cap M=J_M$.
\item The restriction to $J_M$ of $\lambda$ is equal to $\lambda_M$.
\item For any parabolic subgroup $P$ of $G$ with Levi factor $M$, unipotent radical $N$, if one write $N^\op$ for the opposite of $N$ then one has that
\[J=(N\cap J)J_M(N^\op\cap J).
\]
\item Both $N\cap J$ and $N^\op\cap J$ are contained in $\ker\lambda$.
\item For any smooth representation $\Pi$ of $G$, the natural map from $V$ to its Jacquet module $\Pi_N$ induces an injection on the $\lambda$-isotypic subspace of $\Pi$.
\end{enumerate}
\end{definition}

The key fact is the following.

\begin{theorem}[\cite{BushnellKutzko1998}]
Suppose that $(J,\lambda)$ is a $G$-cover of the $\frak{s}_M$-type $(J_M,\lambda_M)$. Then $(J,\lambda)$ is an $\frak{s}$-type.
\end{theorem}

\section{The Kim--Yu construction}\label{sec:kimyu}

In this section, we briefly recall the key points of the Kim--Yu construction of types \cite{Yu2001,KimYu2017}. We adhere to the notational conventions of \cite[\S 5]{LathamNevins2020}, and refer to \emph{ibid.} for many of the more technical details---in particular, we do not specify how the construction of Heisenberg--Weil representations works, and we largely ignore the subtlety of generic diagrams of embeddings of buildings;  whenever these issues become relevant, we will refer to the corresponding passages in \emph{ibid.}

The reader who is only interested in reaching the statement of our inertial correspondence in Section \ref{sec:inertial} will be able to get away without knowing the precise details of the construction discussed during this section. However, these become essential for the results of Section \ref{sec:construction}.

\begin{definition}
A \emph{datum} is a tuple $\Sigma=(\vec{G},M^0,u,\sigma,\vec{r},\vec{\phi})$, consisting of the following objects.
\begin{enumerate}[(i)]
\item A sequence $\vec{G}=(G^0\subsetneq G^1\subsetneq\cdots\subsetneq G^d=G)$ of tame twisted Levi subgroups of $G$, i.e. subgroups $G^i$ of $G$ which are of the form $G^i=\bG^i(F)$, where $\bG^i\subset\bG$ is a connected reductive subgroup for which there exists a finite tamely ramified extension $E/F$ such that $\bG^i\times_FE$ is a Levi subgroup of $\bG\times_FE$.
\item A (genuine) Levi subgroup $M^0$ of $G^0$.
\item A vertex $u$ contained in $\buil(M^0)$.
\item An irreducible representation $\sigma$ of $G_{u}^0$ such that $(G_u^0,\sigma)$ is an $\frak{s}^0$-type for some inertial equivalence class $\frak{s}^0\in\frak{B}(G^0)$ which consists of depth-zero supercuspidal representations.
\item A sequence $\vec{r}=(0<r_0<r_1<\cdots<r_{d-1}=r_d)$ of real numbers. We set $s_i=r_i/2$ for each $i$.
\item A sequence $\vec{\phi}=(\phi^0,\phi^1,\dots,\phi^d)$ of characters $\phi^i$ of $G^i$, where $\phi^i$ is $G^{i+1}$-generic of depth $r_i$, in the sense of \cite[3.9]{HakimMurnaghan2008}. Moreover, we take $\phi^d$ to be the trivial character.
\end{enumerate}

A \emph{truncated datum} is a tuple of the form $(\vec{G},M^0,u,\vec{r},\vec{\phi})$, i.e. a datum without any specified representation $\sigma$ of $G_x^0$.
\end{definition}

To illuminate the requirement of $\sigma$ set out in point (iv), note that $G_{u}^0/G_{u,0+}^0$ identifies naturally with the points of a disconnected algebraic group over the residue field of $F$, the connected component of which is $G_{u}^0/G_{u,0+}^0$. The requirements of (iv) are equivalent to asking that $\ker\sigma\supset G_{u,0+}^0$ and that, when one views $\sigma$ as a representation of $G_u^0/G_{u,0+}^0$ and restricts it to $G_{u,0:0+}^0$, it becomes isomorphic to a sum of pairwise $G_u^0$-conjugate cuspidal representations of $G_{u,0:0+}^0$.

\begin{remark}
As discussed at length in \cite{KimYu2017}, in order for this definition to work out entirely as desired, one must also specify a \emph{generic diagram of embeddings of buildings}. This is a highly technical notion which will merely play a background role in results we cite. However, it is not difficult to see that there exist \emph{many} such generic diagrams (in fact, almost all choices of embeddings lead to such diagrams). Due to this, we omit it from our notation and refer the reader to \emph{ibid.} for a discussion; it is to be understood throughout that we have fixed a generic diagram of embeddings.
\end{remark}

To such a diagram one associates a great many groups and representations, as detailed in \cite[\S 5]{LathamNevins2020}. For the purposes of this paper, we require the following three groups:
\begin{align*}
J(\Sigma)&=G_{x}^0G_{x,s_{0}}^1\cdots G_{x,s_{d-1}}^d;\\
J_+(\Sigma)&=G_{x:0+}^0G_{x,s_0}^1\cdots G_{x,s_{d-1}}^d;\quad\text{and}\\
H_+(\Sigma)&=G_{x,0+}^0G_{x,s_0+}^1\cdots G_{x,s_{d-1}+}^d.
\end{align*}
Then one has a chain of inclusions $H_+(\Sigma)\subset J_+(\Sigma)\subset J(\Sigma)$, and $H_+(\Sigma)$ and $J_+(\Sigma)$ are both pro-$p$ normal subgroups of $J(\Sigma)$. Moreover, there is a canonical isomorphism $J(\Sigma)/J_+(\Sigma)=G_{u}^0/G_{u,0+}^0$, i.e. we may identify $\sigma$ with an irreducible representation of $J(\Sigma)$ whose kernel contains $J_+(\Sigma)$.

After restricting each character $\phi^i$ to $G_{u,s_{i}+}^i$, one may extend these restricted characters to obtain a sequence $\hat{\phi}^i$ of characters, each of which is defined on $H_+(\Sigma)$; we set $\theta_\Sigma=\bigotimes_{i=0}^d\ \hat{\phi}_i$. We refer to $\theta_\Sigma$ as the \emph{semisimple character} associated to the datum $\Sigma$.

For each character $\hat{\phi^i}$ there exists a canonical (and essentially unique; see \cite[\S 3.3]{HakimMurnaghan2008}) irreducible representation $\kappa_\Sigma^i$ of $J(\Sigma)$ which is isomorphic to a sum of copies of $\hat{\phi}^i$ upon restriction to $H_+(\Sigma)$. The representation $\kappa_\Sigma^i$ is called the \emph{Heisenberg--Weil lift} of $\hat{\phi}_i$. Moreover, if we define $\kappa_\Sigma=\bigotimes_{i=0}^d\kappa_\Sigma^i$, then $\kappa_\Sigma$ is an irreducible representation of $J(\Sigma)$, the restriction to $H_+(\Sigma)$ of which is $\theta_\Sigma$-isotypic.

\begin{definition}
We set $\lambda_\Sigma=\sigma\otimes\kappa_\Sigma$.
\end{definition}

\begin{theorem}[\cite{KimYu2017,Fintzen2018}]
The representation $\lambda_\Sigma$ of $J(\Sigma)$ is irreducible, and there exists an inertial equivalence class $\frak{s}\in\frak{B}(G)$ such that $(J(\Sigma),\lambda_\Sigma)$ is an $\frak{s}$-type.

Moreover, under our hypothesis that $p$ is coprime to the order of the absolute Weyl group of $\bG$, every inertial equivalence class $\frak{s}\in\frak{B}(G)$ admits a type of the form $(J(\Sigma),\lambda_\Sigma)$ for some datum $\Sigma$.
\end{theorem}

The construction of \cite{KimYu2017} actually produces covers, in the sense of the previous section. Given a datum $\Sigma=(\vec{G},M^0,u,\sigma,\vec{r},\vec{\phi})$, one can write $M$ for the $G$-centralizer of the centre of $M^0$; this is a Levi subgroup of $G$. The restriction of $\lambda_\Sigma$ to $J(\Sigma)_M:=J(\Sigma)\cap M$ is irreducible, and $(J(\Sigma),\lambda_\Sigma)$ is a $G$-cover of $(J(\Sigma)_M,\lambda_\Sigma)$.

For convenience, we will equip the set of all types of the form $(J(\Sigma),\lambda_\Sigma)$ with a very simple equivalence relation---we say that $(J(\Sigma),\lambda_\Sigma)$ and $(J(\Sigma'),\lambda_{\Sigma'})$ are equivalent if they are both types for the same inertial equivalence class, i.e. if one has that $\Rep_{\lambda_\Sigma}(G)=\Rep_{\lambda_{\Sigma'}}(G)$. It is worth remarking that, via the results of \cite{HakimMurnaghan2008}, we know rather more about what the ``correct'' equivalence relation to impose on the set of data is , but the above equivalence relation is the natural one for our intended application.

\section{$L$-parameters}\label{sec:l-params}

We fix, once and for all, a separable algebraic closure $\bar{F}/F$, and write $W_F\subset\Gal(\bar{F}/F)$ for the Weil group and $I_F\subset W_F$ for the inertia subgroup. Write $\bG^\vee$ for the complex reductive dual of $\bG$, and $G^\vee=\bG^\vee(\C)$; the Langlands dual group is then the semidirect product $^LG=W_F\ltimes G^\vee$, containing $G^\vee$ as its identity component. In particular, since $G^\vee$ is normal in $^LG$, one obtains a natural action of $W_F$ on $G^\vee$.

We will also require the \emph{Weil--Deligne group}. This is the semidirect product $W_F':=W_F\ltimes\C$, where $w\in W_F$ acts on $x\in\C$ as $wxw^{-1}=\|w\|x$. Observe that we have a diagram of inclusions:
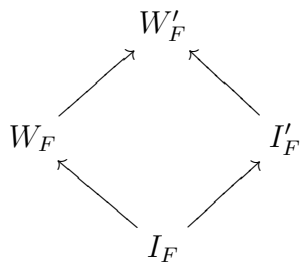
\[\xymatrix{&W_F'&\\
W_F\ar[ur]&&I_F'\ar[ul]\\
&I_F\ar[ur]\ar[ul]&
}\]
\begin{definition}
An \emph{$L$-parameter} for $G$ is an element of the continuous cohomology set $\rH^1(W_F',G^\vee)$.\footnote{More properly, such an element is an \emph{equivalence class} of $L$-parameters; since everything we will consider will be invariant under change of representative of the equivalence class, this viewpoint is most convenient for our purposes.}

Note that when we say \emph{continuous} cohomology, we mean the set of cocycles which are smooth when restricted to $W_F$ and Zariski-continuous when restricted to $\C$, modulo coboundaries.
\end{definition}

Dual to the above diagram, we obtain a diagram of canonical restriction maps:
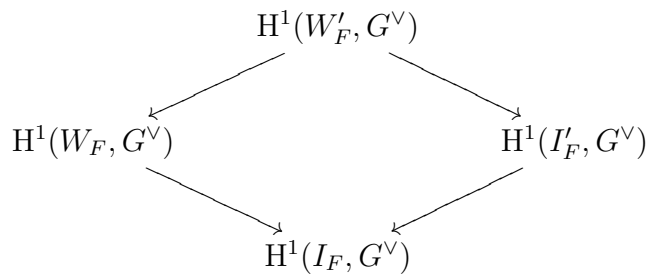
\[\xymatrix{&\rH^1(W_F',G^\vee)\ar[dl]\ar[dr]&\\
\rH^1(W_F,G^\vee)\ar[dr]&&\rH^1(I_F',G^\vee)\ar[dl]\\
&\rH^1(I_F,G^\vee)&
}\]
Much of the content of this paper will be focused on imposing various equivalence relations on the set of irreducible representations of $G$ via this diagram, and understanding the relationships between these diagrams. In particular, the local Langlands correspondence is expected to give a finite-to-one parametrization $\rec:\Irr(G)\rightarrow\rH^1(W_F',G^\vee)$; the finite fibres of this map give a partition of $\Irr(G)$ into \emph{$L$-packets}. One can alternatively consider the fibres of $\varphi\circ\rec$, where $\varphi$ is any of the maps in the above diagram. Many of our questions will be aimed towards giving a precise description of the fibres of the resulting equivalence relation; our ultimate goal is the construction of a set of objects which naturally parametrize the fibres of $\Res_{I_F'}^{W_F'}\circ\rec$, which is essentially achieved by chasing around the diagram in the opposite direction.

The precise role of the action of $\C$ associated to an $L$-parameter will play a particularly key role throughout the paper; we refer to this action as the \emph{monodromy} of an $L$-parameter.

There are another family of useful maps on these cohomology sets which will be of use to us. Given a closed subgroup $\bH$ of $\bG$, there is often a dual embedding $\iota:H^\vee\hookrightarrow G^\vee$, which extends to an embedding $\iota:{}^LH\hookrightarrow{}^LG$. Via this embedding, one obtains a pushforward $\iota_*:\rH^1(\Gamma,H^\vee)\rightarrow\rH^1(\Gamma,G^\vee)$ whenever $\Gamma$ is one of the above groups. While one requires some hypotheses on $\bH$ in order to guarantee that the dual embedding works as desired, it is known that one can construct such a dual embedding when $\bH$ is either a torus or a Levi subgroup of $\bG$; this will be sufficient for our purposes. A description of the construction of these dual embeddings is given in \cite[5.1]{Kaletha2019}.

\section{Kaletha's correspondence}\label{sec:kaletha}

In \cite{Kaletha2019}, Kaletha directly constructs an extremely strong candidate for the image under the local Langlands correspondence of the set of \emph{regular supercuspidal representations} of $G$. While it is \emph{not} known that Kaletha's construction \emph{is} the local Langlands correspondence (due to the fact that it is currently unclear exactly which properties should characterize the correspondence for arbitrary groups), it is widely expected that this is the case. For the remainder of the paper, we assume that this is so.

\begin{definition}
Let $\pi$ be an irreducible representation of $G$ of inertial support $\frak{s}$, and let $(J(\Sigma),\lambda_\Sigma)$ be an $\frak{s}$-type arising from a datum $\Sigma=(\vec{G},M^0,u,\sigma,\vec{r},\vec{\phi})$ as in Section \ref{sec:kimyu}. We say that $\pi$ is \emph{regular} if the depth-zero component $\sigma$ is regular in the following sense: the restriction of $\sigma$ to $G_{u,0}^0$ is a sum of pairwise $G_{[u],0}^0$-conjugate representations; fix one such component $\sigma_0$. Then $\sigma_0$ is inflated from a cuspidal irreducible representation of $G_{u,0:0+}^0$, a finite group of Lie type; we say that $\sigma$ is regular if $\sigma_0$ is a regular Deligne--Lusztig representation, i.e. there exists a minisotropic maximal torus $\mathsf{T}\subset G_{u,0:0+}^0$, and a character $\vartheta$ of $\mathsf{T}$ in general position, such that $\sigma_0=\pm\mathrm{R}_{\mathsf{T}}\theta$.
\end{definition}

Kaletha shows that, when $\pi$ is a regular \emph{supercuspidal} representation of $G$, the information in the datum $\Sigma$ may be encoded in a \emph{regular elliptic pair} $(S,\Theta)$; this is a pair consisting of an elliptic torus $S$, together with a character $\Theta$ of $S$ which satisfies a number of properties. In doing so, Kaletha obtains a canonical bijection between the set of regular supercuspidal representations of $G$ and the set of $G$-conjugacy classes of regular elliptic pairs $(S,\Theta)$, which we denote by $\pi_{S,\Theta}\leftrightarrow(S,\Theta)$. For the most part, we are able to treat this construction as a black box, although we do require one basic property: by following Kaletha's construction of the bijection $(S,\Theta)\rightarrow\pi_{S,\Theta}$, which proceeds by first associating a datum $\Sigma_{S,\Theta}$ to $(S,\Theta)$, it is apparent that if $\chi$ is a character of $G$ such that $(S,\Theta\otimes\chi)$ is also a regular elliptic pair, then one has $\pi_{S,\Theta\otimes\chi}=\pi_{S,\Theta}\otimes\chi$. In particular, this is always the case when $\chi$ is an unramified character of $G$.

We write $\Cusp(G)$ for the set of isomorphism classes of regular supercuspidal representations, i.e. the set of isomorphism classes of representations of the form $\pi_{S,\Theta}$ for some regular elliptic pair $(S,\Theta)$. More generally, we will make use of the following definition throughout the remainder of the paper.

\begin{definition}
An irreducible representation $\pi$ of $G$ is \emph{regular} if its supercuspidal support $\sc(\pi)$ is a regular representation of some Levi subgroup of $G$.

We write $\Reg(G)$ for the set of isomorphism classes of regular representations of $G$.
\end{definition}

\begin{remark}
This is a slightly non-standard notion of regularity within the context of the Langlands programme. More typically, there is a notion of a regular $L$-parameter, and one says that $\pi$ is regular if the local Langlands correspondence assigns to it a regular $L$-parameter. It is widely expected that representations which are regular in our sense should also be regular in this sense, and our notion of regularity is much more natural from the $p$-adic group side of the correspondence, which is where much of our work will take place.
\end{remark}

The major result of \cite{Kaletha2019} is the construction of a local Langlands correspondence for the regular supercuspidal representations.

\begin{theorem}[\cite{Kaletha2019}]
There exists a natural map $\rec_G:\Cusp(G)\rightarrow\rH^1(W_F,G^\vee)$.
\end{theorem}

As discussed above, the remainder of the paper assumes that this map $\rec_G$ \emph{is} the local Langlands correspondence.

\begin{remark}
Notice that the image of this map is contained in $\rH^1(W_F,G^\vee)$, \emph{not} in $\rH^1(W_F',G^\vee)$---in other words, every regular supercuspidal representation of $G$ corresponds to a parameter with a trivial monodromy action.
\end{remark}

As a consequence of this construction, one easily obtains an inertial Langlands correspondence. We write $\Type(G)$ for the set of equivalence classes of cuspidal Kim--Yu types which are contained in regular supercuspidal representations, i.e. the set of pairs of the form $(J(\Sigma),\lambda_\Sigma)$ where $\Sigma$ is some datum with regular Deligne--Lusztig depth-zero component. There is therefore a canonical surjective map $\Cusp(G)\rightarrow\Type(G)$, the image of which we denote by $\Type_\sc(G)$.

\begin{corollary}
There exists a natural map $\iner_G:\Type_\sc(G)\rightarrow\rH^1(I_F,G^\vee)$ such that the following diagram commutes.
\[\xymatrix{
\Cusp(G)\ar[r]^-{\rec_G}\ar[d] & \rH^1(W_F,G^\vee)\ar[d]^{\Res_{I_F}^{W_F}}\\
\Type_\sc(G)\ar[r]_-{\iner_G} & \rH^1(I_F,G^\vee)
}\]
\end{corollary}

\begin{proof}
Given some $(J,\lambda)$ representing a class in $\Type(G)$, since $(J,\lambda)$ is an $\frak{s}$-type for some inertial equivalence class $\frak{s}=[G,\pi]_G$, any two irreducible representations of $G$ which contain $\lambda$ are related by a twist by some unramified character of $G$. In particular, the composition $\Res_{I_F}^{W_F}\rec(\pi)$ is independent of the choice of $\pi$. This gives the desired map $\iner_G:\Type_\sc(G)\rightarrow\rH^1(I_F,G^\vee)$.
\end{proof}

\section{Parabolic induction}\label{sec:non-cusp}\label{sec:parabolic}

Next, we explain how Kaletha's construction is \emph{expected} to extend to a local Langlands correspondence for $\Reg(G)$. We caution the reader that, in the large majority of cases, this has not been proved and is entirely conjectural.

\begin{hypothesis}\label{hyp}
Kaletha's map $\rec_G:\Cusp(G)\rightarrow\rH^1(W_F,G^\vee)$ extends to a map $\rec_G:\Reg(G)\rightarrow\rH^1(W_F',G^\vee)$ such that, for any Levi subgroup $M$ of $G$, and any dual embedding $\iota:{}^LM\hookrightarrow{}^LG$, if one denotes by $\Reg_M(G)$ the set of isomorphism classes of irreducible representations with supercuspidal support defined on $M$, then the following diagram is commutative.
\[\xymatrix{\Reg_M(G)\ar[d]_{\sc}\ar[r]^-{\rec_G}&\rH^1(W_F',G^\vee)\ar[d]^{\Res_{W_F}^{W_F'}}\\
\Cusp(M)\ar[r]_-{\rec_M}&\rH^1(W_F,M^\vee)
}\]
\end{hypothesis}
In particular, note that from this implies that if $\pi$ has supercuspidal support $(M,\rho)$, then $\Res_{W_F}^{W_F'}\ \rec_G(\pi)\in\iota_*\rH^1(W_F,M^\vee)$ .

\begin{proposition}[The weak inertial Langlands correspondence]
There exists a map $\overline{\iner}_G:\Type(G)\rightarrow\rH^1(I_F,G^\vee)$ which makes the following diagram commute.
\[\xymatrix{
\Reg(G)\ar[rrr]^-{\rec_G}\ar[dr]_{\sc}\ar[ddd] & & & \rH^1(W_F',G^\vee)\ar[ddd]^-{\Res_{I_F}^{W_F'}}\ar[dl]^{\Res_{W_F}^{W_F'}}\\
& \coprod_{M\subset G}\Cusp(M)\ar[r]^-{\iota_*\rec_M}\ar[d]&\rH^1(W_F,G^\vee)\ar[d]^{\Res_{I_F}^{W_F}}&\\
& \coprod_{M\subset G}\Type_\sc(M)\ar[r]_-{\iner_M}&\rH^1(I_F,G^\vee) &\\
\Type(G)\ar[ur]\ar[rrr]_-{\overline{\iner}_G}& & &\rH^1(I_F,G^\vee)\ar@{=}[ul]
}\]
Here, the map $\Type(G)\rightarrow\coprod\Type(M)$ is the canonical one obtained by recognizing that each $(J,\lambda)$ is a $G$-cover of $(J_M,\lambda_M)$ for some $M$; one can therefore map $(J,\lambda)$ to $(J\cap M,\lambda)$.
\end{proposition}

\begin{proof}
This follows easily from the supercuspidal case, given the properties listed in Hypothesis \ref{hyp}.
\end{proof}

\begin{remark}
We refer to $\overline{\iner}_G$ as the \emph{weak} inertial Langlands correspondence, since it is not capable of distinguishing between representations with equivalent supercuspidal supports; on the Galois side of the correspondence, this is the same as noting that its image consists of classes with trivial monodromy. Our goal for the remainder of the paper will be to remedy this, realizing $\overline{\iner}_G$ as a quotient of a ``strong'' inertial Langlands correspondence $X\rightarrow\rH^1(I_F',G^\vee)$, where $X$ is a set of objects which is closely related to $\Type(G)$.
\end{remark}

\section{Typical representations}\label{sec:typical}

In order to construct our strong form of the inertial Langlands correspondence, we will need to work with a slightly more general class of objects than $\frak{s}$-types.

\begin{definition}
Let $(J,\lambda)$ be a pair consisting of an irreducible representation $\lambda$ of a compact open subgroup $J$ of $G$. We say that $(J,\lambda)$ is $\frak{s}$-typical if, whenever $\pi\in\Irr(G)$ contains $\lambda$, one has $\pi\in\Rep^{\frak{s}}(G)$.
\end{definition}

\begin{remark}
In other words, this is close to asking that $\Rep_\lambda(G)$ is a subcategory of $\Rep^{\frak{s}}(G)$ (as opposed to being equal to $\Rep^{\frak{s}}(G)$, as would be the case if $(J,\lambda)$ were an $\frak{s}$-type). However, showing that this is equivalent to the above definition encounters some difficulties when one considers inadmissible representations of $G$, since if $(J,\lambda)$ is $\frak{s}$-typical but not an $\frak{s}$-type, then $\Rep_\lambda(G)$ is not a Serre subcategory of $\Rep(G)$---in this case $\Rep_\lambda(G)$ is not closed under formation of subquotients.
\end{remark}

As long as one has a construction of types, it is easy to see that many typical representations exist.

\begin{lemma}
Suppose that $(J,\lambda)$ is an $\frak{s}$-type and that $K\supset J$ is a maximal compact subgroup of $G$.
\begin{enumerate}[(i)]
\item For any irreducible component $\tau$ of $\Ind_J^K\ \lambda$, the pair $(K,\tau)$ is $\frak{s}$-typical.
\item If $\pi\in\Irr(G)$ contains $\lambda$, then there exists a pair $(K,\tau)$ as above such that $\pi|K$ contains $\tau$.
\end{enumerate}
\end{lemma}

\begin{proof}
Each of these is a straightforward consequence of Frobenius reciprocity.
\end{proof}

In particular, one obtains a collection of typical representations which are of the form $(K,\tau)$ where $K\subset G$ is a maximal compact subgroup and $\tau|_{J(\Sigma)}$ contains $\lambda_\Sigma$ for some datum $\Sigma$ such that $J(\Sigma)\subset K$. We denote the set of $G$-conjugacy classes of such typical representations by $\Typ(G)$. Notice that there is a canonical map $\Typ(G)\rightarrow\Type(G)$, given by mapping a typical representation to its underlying equivalence class of types.

While there is not a canonical map $\Irr(G)\rightarrow\Typ(G)$, there \emph{is} a canonical map from $\Irr(G)$ to the power set $\mathcal{P}\Typ(G)$, namely the map which sends $\pi$ to the set $X_\pi$ of $G$-conjugacy classes of $\frak{s}$-typical representations $(K,\tau)$ such that $\Hom_K(\tau,\pi)\neq 0$, as $K$ ranges over a set of representatives of the conjugacy classes of maximal compact subgroups of $G$. Moreover, the canonical map $\Typ(G)\rightarrow\Type(G)$ clearly extends to a partial map $\mathcal{P}\Typ(G)\rightarrow\Type(G)$, defined on the subset of $\mathcal{P}\Typ(G)$ consisting of sets of the form $X_\pi$.

\section{The inertial Langlands correspondence}\label{sec:inertial}

With the above discussion in mind, we write $\varphi$ for the canonical map $\Reg(G)\rightarrow\mathcal{P}\Typ(G)$, defined by $\varphi(\pi)=X_\pi$. We write $\Typ_\varphi(G)$ for the image of $\varphi$.

The key to constructing a ``strong'' form of the inertial Langlands correspondence is the following result.

\begin{theorem}\label{thm1}\label{thm:main}
Suppose that $\pi_1,\pi_2\in\Reg(G)$ are such that $\sc(\pi_1)=\sc(\pi_2)$ and $\varphi(\pi_1)=\varphi(\pi_2)$. Then $\pi_1\simeq\pi_2$.
\end{theorem}

\begin{proof}
Let $(M,\rho)$ denote a representative of $\sc(\pi_1)=\sc(\pi_2)$, and choose a parabolic subgroup $P$ of $G$ with Levi factor $M$, so that both $\pi_1$ and $\pi_2$ occur as irreducible subquotients of the representation $\Ind_P^G\ \rho$. It is clear that $\pi_1\simeq\pi_2$ when $M=G$, so we may assume that $M$ is a proper Levi subgroup of $G$.

Fix a datum $\Sigma=(\vec{G},M^0,u,\sigma,\vec{r},\vec{\phi})$ such that $(J(\Sigma),\lambda_\Sigma)$ is an $[M,\rho]_M$-type; then the fact that $M\neq G$ implies that $M^0\neq G^0$, and so $u$ is contained in the interior of a facet of positive dimension in $\buil(G^0)$, and hence also in the interior of some positive-dimensional facet $\mathscr{F}\subset\buil(G)$. By \cite[6.3]{MoyPrasad1996}, we may associate to $\mathscr{F}$ a Levi subgroup $M_\mathscr{F}$ of $G$ lying in the $G$-orbit of $M$; without loss of generality we may take $M=M_\mathscr{F}$. Fix an apartment $\apart\subset\buil(G)$ containing $\mathscr{F}$ , corresponding to a maximal split torus $S\subset G$.

Let $y_1,\dots,y_n$ denote the vertices of $\mathscr{F}$, so that one has $J(\Sigma)\subset G_{y_i}$ for each $i$. Since $M$ contains $U_\psi$ whenever $\psi\in\Psi(G,S)$ is an affine root which takes constant value on $\mathscr{F}$, one easily sees that $G_{y_i}P=G_{y_j}P$ for all $i,j$, from which it follows that $G_{y_i}P=G_{\mathscr{F}}P$ for all $i$.

For each index $i$, let $V_i$ denote the maximal multiplicity-free $\lambda_\Sigma$-isotypic subrepresentation of $\pi_1|_{G_{y_i}}$. Then by assumption we have injective $G_{y_i}$-equivariant maps $V_i\rightarrow\pi_2$. We may therefore define subrepresentations $W_i$ of $\Ind_P^G\ \rho$ generated by $V_i$ for each $i$ and observe that both $\pi_1$ and $\pi_2$ are contained in $W:=\bigcap_i W_i$. We claim that this representation $W$ must be irreducible.

Since the type $(J(\Sigma),\lambda_\Sigma)$ is a $G$-cover of the $[M,\rho]_M$-type $(J(\Sigma)\cap M,\lambda_\Sigma)$, there is an Iwahori decomposition $J(\Sigma)=(J\cap\bar{N})(J\cap P)$, where $\bar{N}$ denotes the unipotent radical of the parabolic subgroup $\bar{P}$ such that $\bar{P}\cap P=M$. Furthermore, the representation $\Ind_P^G\ \rho$, being contained in the Bernstein block $\Rep^{[M_\rho]_G}(G)$, is generated by its $\lambda_\Sigma$-isotypic vectors. We may therefore extend $\rho$ to a representation of the group generated by $J(\Sigma)P$ by declaring that $J(\Sigma)\cap \bar{N}$ acts trivially. As $J(\Sigma)\subset G_{y_i}$ for each $i$ and each representation $V_i$ is $\lambda_\Sigma$-isotypic, we find that $W$ is contained in the representation $\Ind_{\langle J(\Sigma)P\rangle}^G\ \rho$, which canonically occurs as a subrepresentation of $\Ind_P^G\ \rho$. 

Again using the fact that $J(\Sigma)$ is a $G$-cover of $(J(\Sigma)\cap M,\lambda_\Sigma)$, as $M$ is the Levi subgroup associated to $\mathscr{F}$ it follows from the fact that $\Ind_{J(\Sigma)\cap M}^{M_{\mathscr{F}}}\ \lambda_\Sigma$ is irreducible that $\Ind_{J(\Sigma)}^{G_\mathscr{F}}\ \lambda_\Sigma$ is irreducible, and hence so is $\tilde{\rho}:=\Ind_{\langle J(\Sigma)P\rangle}^{\langle G_\mathscr{F}P\rangle}\ \rho$. In particular, we find that $W$ is contained in $\Ind_{\langle G_\mathscr{F}P\rangle}^G\ \tilde{\rho}$, and that this representation is identical to $\Ind_{\langle G_{y_i}P\rangle}^G\ \tilde{\rho}$. In particular, applying Frobenius reciprocity we conclude that the representations $\pi_1$ and $\pi_2$ of $G$ must intertwine when viewed as representations of $\langle G_{y_i}P\rangle$. It remains for us to show that $\langle G_{y_i}P\rangle=G$.

Fix a maximal split torus $S\subset G$ which contains the centre of $M$, and which is such that the corresponding apartment $\apart\subset\buil(G)$ contains the facet $\mathscr{F}$, and hence each of the points $y_i$. Let $W_S^\aff=N_G(S)/S_{\mathrm{b}}$ denote the affine Weyl group of $S$; here, $S_{\mathrm{b}}\subset S$ denotes the maximal bounded subgroup. For each point $v\in\apart$, we may define a subgroup $W_v$ of $W_S^\aff$ generated by the reflections $s_\psi$ for those affine roots $\psi\in\Psi(S)$ such that $v$ lies on the hyperplane $H_\psi=\{v'\in\apart\ |\ \psi(v')=0\}$. Each group $W_v$ is a finite group which, by pinning the root system $\Phi(G,S)$ at a special vertex of $\mathscr{C}$, may be identified with a finite subgroup of the Weyl group $W_S=N_G(S)/S$. Clearly $W_v$ fixes the simplicial closure of $v$ pointwise, and so $W_v\subset G_v$.

Let $H=\langle G_z\ |\ z\in\mathscr{F}\rangle=\langle G_{y_i}\ |\ i\rangle$. Fix a chamber $\mathscr{C}$ in $\apart$ containing $\mathscr{F}$. The chamber $\mathscr{C}$ is a polysimplex, and therefore admits a decomposition $\mathscr{C}=\mathscr{C}_1\times\cdots\times\mathscr{C}_k$ as a product of simplices. Since $M$ is a proper Levi subgroup of $G$, the facet $\mathscr{F}$ contains at least two vertices of $\mathscr{C}$, and so in the case that $k=1$ we find that every wall of $\mathscr{C}$ must contain a vertex of $\mathscr{F}$. Then $H$ contains $s_\psi$ whenever $H_\psi$ is a wall of $\mathscr{C}$. From this, a straightforward inductive argument shows that the same conclusion holds true regardless of the value of $k$. 

So $H$ contains $G_{y_i}wG_{y_i}$ whenever $y_i$ is a vertex of $\mathscr{C}$ and $w\in W_S^\aff$, hence it also contains $\bigcup_{w\in W_S^\aff}G_\mathscr{C}wG_\mathscr{C}$; by the Bruhat--Tits decomposition this union contains every maximal compact subgroup of $G$. In particular, it contains $G_z$ for some special vertex $z\in\mathscr{C}$. From this, we see that $\langle G_{y_i}P\rangle=\langle G_zP\rangle=G_zP=G$, by the Iwasawa decomposition. It follows that both $\pi_1$ and $\pi_2$ both admit injective $G$-equivariant maps into the irreducible representation $\Ind_{G_\mathscr{F}P}^G\ \rho$, and are therefore isomorphic, as desired.
\end{proof}

With this, we come to our main result.

\begin{theorem}[The inertial Langlands correspondence]\label{thm:inertial}
There exists a map $\iner_G:\Typ_\varphi(G)\rightarrow\rH^1(I_F',G^\vee)$ such that the following diagram commutes.
\[\xymatrix{
\Reg(G) \ar[r]^-{\rec_G}\ar[d]_\varphi & \rH^1(W_F',G^\vee)\ar[d]^{\Res_{I_F'}^{W_F'}}\\
\Typ_\varphi(G)\ar[r]_-{\iner_G} & \rH^1(I_F',G^\vee)
}\]
Moreover, the map $\iner_G$ is compatible with $\overline{\iner}_G$, in the sense that one has a commutative diagram
\[\xymatrix{
\Typ_\varphi(G)\ar[r]^-{\iner_G}\ar[d]&\rH^1(I_F',G^\vee)\ar[d]^{\Res_{I_F}^{I_F'}}\\
\Type(G)\ar[r]_-{\overline{\iner}_G} & \rH^1(I_F,G^\vee)
}\]
\end{theorem}

\begin{proof}
Given $X\in\Typ_\varphi(G)$, choose a representation $\pi\in\Reg(G)$ such that $\varphi(\pi)=X$. Then we set $\iner_G(X)=\Res_{I_F'}^{W_F'}\rec_G(\pi)$.

Suppose that $\pi'\in\Reg(G)$ also satisfies $\varphi(\pi')=X$. We claim that $\Res_{I_F'}^{W_F'}\rec_G(\pi')=\Res_{I_F'}^{W_F'}\rec_G(\pi)$. Since $\varphi(\pi)=\varphi(\pi')$, projecting along the canonical map $\Typ_\varphi(G)\rightarrow\Type(G)$ shows that $\pi$ and $\pi'$ are inertially equivalent. In other words, if $(M,\rho)$ is a representative of $\sc(\pi)$, then there exists an unramified character $\omega$ of $M$ such that $(M,\rho\otimes\omega)$ is a representative of $\sc(\pi')$. Since $\pi,\pi'\in\Reg(G)$, by definition we may find regular elliptic pairs $(S,\Theta)$ and $(S',\Theta')$ in $M$ such that $\rho=\pi_{S,\Theta}$ and $\rho'=\pi_{S',\Theta'}$; since $\rho'=\rho\otimes\omega=\pi_{S,\Theta}\otimes\omega=\pi_{S,\Theta\omega}$, it follows that we may without loss of generality take $S=S'$ and $\Theta'=\Theta\omega$.

Now consider the $L$-parameter $\rec_G(\pi')$. By construction, the semisimplification of this is precisely the pushforward $\iota_*\rec_M(\rho')=\iota_*\rec_M(\rho\otimes\omega)$ along an appropriate embedding $\iota:{}^LM\hookrightarrow{}^LG$. Since both $\rho$ and $\rho'=\rho\otimes\omega$ are constructed from regular elliptic pairs defined on the same torus $S\subset M$, one finds that both $\iota_*\rec_M(\rho)$ and $\iota_*\rec_M(\rho')$ have image in $^LS\hookrightarrow{}^LM\hookrightarrow{}^LG$.

Given an element $a$ of the image of $\rec_G(\pi')$, one may take the Jordan decomposition $a=su$ with $s$ semisimple and $u$ unipotent, so that $s$ is contained in the image of $\iota_*\rec_M(\rho')$, which is an abelian group. If $a=\rec_G(\pi')(w)$, then $a$ must commute with the semisimple element $b=\iota_*\rec_T(\omega)(w)$, and one finds that $(sb)u=(bs)u=u(bs)=u(sb)$ is a Jordan decomposition of $ab$. In particular, this implies that $\red_T(\omega)$ extends to a 1-cocycle $c_\omega:W_F'\rightarrow G^\vee$ which is trivial on both $I_F$ (since $\omega$ is unramified) and $\C$. Given a 1-cocycle $f$ representing the class $\rec_G(\pi)$, the pointwise product $fc_\omega$ is a 1-cocycle whose cohomology class is independent of the choice of $f$; in this way we obtain a class $\rec_G(\pi)\cdot c_\omega\in\rH^1(W_F',G^\vee)$; by the above observations on the Jordan decomposition, we see that one must have $\rec_G(\pi)\cdot c_\omega=\rec_G(\pi')$.

We may therefore assume without loss of generality that $\sc(\pi)=\sc(\pi')$, in which case the assumption that $\varphi(\pi)=\varphi(\pi')$ allows us to apply Theorem \ref{thm1} to deduce that $\pi\simeq\pi'$. As argued above, changing the supercuspidal support of $\pi'$ by an unramified twist does not affect $\Res_{I_F'}^{W_F'}\rec_G(\pi')$, and so we conclude that $\Res_{I_F'}^{W_F'}\rec_G(\pi)$ is independent of the choice of $\pi$. In particular, this gives us our definition of $\iner_G:\Typ_\varphi(G)\rightarrow\rH^1(I_F',G^\vee)$, from which the commutativity of the first diagram is immediate.

The commutativity of the second diagram follows easily from Hypothesis \ref{hyp}, along with the fact that the canonical map $\Typ_\varphi(G)\rightarrow\Type(G)$ simply maps each typical representation to its underlying type.
\end{proof}

\begin{remarks}
One should note that, due to the hypotheses on $p$ which are in force throughout the paper, we only claim to establish the existence of this inertial correspondence for the so-called \emph{essentially tame representations} (which should be those corresponding to $L$-parameters for which the image of the wild inertia subgroup is abelian). It is worth pointing out the precise extent to which these hypotheses are required for the above proof to work. Most significantly, they guarantee the existence of $\frak{s}$-types for each $\frak{s}$; however, the existence of $\frak{s}$-types is known even for wildly ramified representations in a number of cases, such as for general or special linear groups \cite{BushnellKutzko1993,BushnellKutzko1994}, inner forms of general linear groups \cite{SecherreStevens2008}, and classical groups when $p\neq 2$ \cite{Stevens2008}. Once one has a construction of types, the major thing that one needs to know in order to follow the template of the above proofs is that the type is associated to a point in the building, and that the Levi subgroup associated to this point should be equal to the Levi subgroup defining the corresponding inertial equivalence class; this is the case in the known constructions. The other point at which we make use of our hypotheses is when defining regular representations as representations with regular supercuspidal support in the sense of \cite{Kaletha2019}. Such an approach cannot be taken in the non-tame case, meaning that a more complicated approach to regularity must be taken; it seems unlikely that this would be the source of any major technical issues, however. In short, this appears to be an approach which should go through \emph{mutatis mutandis} in any situation where one has a satisfactory construction of $\frak{s}$-types for each inertial equivalence class $\frak{s}$. At the time of writing, however, the only such available constructions which are uniform in appearance as the group $G$ varies are those leading to tame representations.
\end{remarks}

\begin{remarks}
As mentioned during the introduction, this strong form of the inertial Langlands correspondence has been something of a problematic topic within the literature, with varying statements and a rather complex trail of references to be followed if one wishes to see a proof (which, ultimately, boils down to a lengthy but largely straightforward exercise in the Bernstein--Zelevinsky classification \cite{BernsteinZelevinsky1976}). The case of $\mathbf{GL}_n(G)$ is essentially treated as \cite[Prop. 6.2]{SchneiderZink1999}, although this is presented in a somewhat different language; \cite[Prop. 6.5.3]{BellaicheChenevier2009} describes how one may translate it into a statement closer to the viewpoint taken in this paper, and such an approach is also discussed in \cite[\S 4.1]{EmertonGee2014}.

However, while it appears to be a common expectation that such results should generalize beyond $\mathbf{GL}_n(F)$, to the author's knowledge the only other situation in which this has been fully worked out is that of $\mathbf{SL}_n(F)$ in an unpublished note by Will Conley.  At best, there is a widely believed folklore conjecture which predicts that there should be a canonical map $\Typ(G)\rightarrow\rH^1(I_F',G^\vee)$ (or, more likely, a map defined on some well-chosen subset of $\Typ(G)$). Such a result would be a considerable refinement of the results of this section, but this appears to be considerably beyond the reach of our current knowledge. In particular, any construction of such a map appears to have several serious pre-requisites:
\begin{enumerate}[(i)]
\item one must have a detailed understanding of how parabolically induced representations decompose into irreducible components;
\item one must have a detailed understanding of the set of typical representations; and
\item one must know that there exists a very precise correspondence between the set of typical representations and the set of representations with a given inertial support.
\end{enumerate}
The first of these, in particular, is a serious obstacle, with very little progress having been made beyond the Bernstein--Zelevinsky classification, apart from for certain families of representations of classical groups.

Our approach largely consists of the observation that, by being willing to work with sets of typical representations rather than singletons, it becomes much easier to show that the map $\iner_G$ is well-defined, with this problem reducing to a clear problem about the representation theory of $G$, which we resolve as Theorem \ref{thm:main}.

Nonetheless, one would hope that, at least in a great many cases, it is actually possible to refine the map $\Typ_\varphi(G)\rightarrow\rH^1(I_F',G^\vee)$ to a (partially defined) map $\Typ(G)\rightarrow\rH^1(I_F',G^\vee)$ and work with typical representations one-by-one. It would be extremely interesting to investigate this in some explicit situations, such as for classical groups where many parabolically induced representations have known decompositions.
\end{remarks}

\section{An explicit construction of typical representations}\label{sec:construction}

With the results of the previous section, we have a satisfactory construction of an inertial Langlands correspondence. However, if one wishes to apply this correspondence, then one must understand precisely the set $\Typ(G)$ of typical representations. In this final section, we provide a geometric construction of these typical representations, using Bruhat--Tits theory. 

We fix some notation for the remainder of the  paper. Let $\frak{s}\in\frak{B}(G)$ be an inertial equivalence class consisting of regular representations, and fix a datum $\Sigma=(\vec{G},M^0,u,\sigma,\vec{r},\vec{\phi})$ such that $(J(\Sigma),\lambda_\Sigma)$ is an $\frak{s}$-type. If $K$ is a maximal compact subgroup of $G$ such that $J(\Sigma)\subset K$, then the $\frak{s}$-typical representations of $K$ are (by definition) precisely the irreducible components of $\Ind_{J(\Sigma)}^K\ \lambda_\Sigma$. We will therefore describe each of these components. We make frequent use of results established during \cite[\S 7]{LathamNevins2020}, which will provide the technical tools necessary to carry out our construction.

From now on, let $x\in\buil(G)$ be a point contained in the simplicial closure of $u$, i.e. such that $G_u\subset G_x$. To state the next lemma, we recall the notion of a \emph{generalized $s_i$-facet} from \cite{DeBacker2002}. We say that two points $y,z\in\buil(G)$ lie in the same generalized $s_i$-facet of $\buil(G)$ if one has that $G_{y,s_i}=G_{z,s_i}$ and $G_{y,s_i+}=G_{z,s_i+}$. Note that, by replacing $u$ with a point with the same simplicial closure if necessary, it is always possible to arrange things so that $x$ lies in the same $s_i$-facet as $u$.

\begin{lemma}
Suppose that $x$ is contained in the same $s_i$-facet as $u$ for each $i$. Write $M^{x,0}$ for the Levi subgroup of $G^0$ associated to the point $x$ via \cite[6.3--4]{MoyPrasad1994}. Then $\Sigma_x:=(\vec{G},M^{x,0},x,\vec{r},\vec{\phi})$ is a truncated datum.
\end{lemma}

\begin{proof}
The subtle point underlying this fact is that one may still construct a suitable generic diagram of embeddings; this is addressed in detail during \cite[7.3]{LathamNevins2020}, where our point $u$ plays the same role as the point denoted by $u_\xi$ in \emph{loc. cit.} Our result is slightly simpler than that of \emph{loc. cit.}, in that we wish only to construct a truncated datum, meaning that the details underlying the construction of $\zeta_\xi$ are irrelevant for our purposes. Otherwise, the proof goes through \emph{mutatis mutandis}.
\end{proof}

Since $u$ is contained in $\buil(M^{x,0})$, the group $M_{u,0:0+}^{x,0}=M_{u,0:0+}^0$ identifies with a Levi subgroup of $M_{x,0:0+}^{x,0}$. Moreover, $M_{u,0}^{x,0}/M_{x,0+}^{x,0}$ is a parabolic subgroup of $M_{x,0:0+}^{x,0}$ with $M_{u,0:0+}^{x,0}$ as its Levi factor. In particular, we may consider the parabolically induced representation $V_{x,\sigma}:=\Ind_{M_{u,0}^{x,0}/M_{x,0+}^{x,0}}^{M_{x,0:0+}^{x,0}}\ \sigma$. Let $\zeta$ be an irreducible component of $V_{x,\sigma}$.

Then we may consider the tuple $\Sigma_{x,\zeta}:=(\vec{G},M^0,x,\zeta,\vec{r},\vec{\phi})$. Note that, while $\Sigma_x$ is a truncated datum, $\Sigma_{x,\zeta}$ is \emph{not} a datum unless $x$ is contained in the interior of the same facet as $u$. Nonetheless, we abuse notation slightly and associate groups, semisimple characters and Heisenberg--Weil representations to $\Sigma_{x,\zeta}$ via $\Sigma_x$. That is, we set $J(\Sigma_{x,\zeta})=J(\Sigma_x)$, $H_+(\Sigma_{x,\zeta})=H_+(\Sigma_x)$, $\theta_{\Sigma_{x,\zeta}}=\theta_{\Sigma_x}$, and $\kappa_{\Sigma_{x,\zeta}}=\kappa_{\Sigma_x}$. We also set $\lambda_{\Sigma_{x,\zeta}}=\zeta\otimes\kappa_{\Sigma_{x,\zeta}}$. The pair $(J(\Sigma_{x,\zeta}),\lambda_{\Sigma_{x,\zeta}})$ is always $\frak{s}$-typical, but it need not be an $\frak{s}$-type.

\begin{lemma}
One has $J(\Sigma_{x,\zeta})\supset J(\Sigma)$ and $H_+(\Sigma_{x,\zeta})\subset H_+(\Sigma)$. Moreover, $\theta|_{H_+(\Sigma_{x,\zeta})}=\theta_{\Sigma_{x,\zeta}}$, and the restriction to $J(\Sigma)$ of $\kappa_{\Sigma_{x,\zeta}}$ is $\kappa_\Sigma$-isotypic.
\end{lemma}

\begin{proof}
These claims are established during the course of the proof of \cite[7.5]{LathamNevins2020}; as before, our point plays the role of the point denoted by $u_\xi$ in \emph{loc. cit.}
\end{proof}

In particular, from this we are able to immediately deduce the following.

\begin{lemma}\label{lem:isotypic}
For each pair $(x,\zeta)$ as above, one has that
\[\Hom_{J(\Sigma)}(\lambda_\Sigma,\lambda_{\Sigma_{x,\zeta}})\neq 0.
\]
\end{lemma}

Equivalently, via Frobenius reciprocity, we may identify $\Ind_{J(\Sigma_{x,\zeta})}^K \lambda_{\Sigma_{x,\zeta}}$ with a subrepresentation of $\Ind_{J(\Sigma)}^K\ \lambda_\Sigma$ whenever $K\supset J(\Sigma_{x,\zeta})\supset J(\Sigma)$.

We equip the set of pairs $(x,\zeta)$ as above with a partial order $\preceq$, where we declare that $(x,\zeta)\preceq(x',\zeta')$ if and only if $x'$ is contained in the simplicial closure of $x$ and $\zeta'$ is contained in $V_{x,\zeta}$. As a consequence of Lemma \ref{lem:isotypic}, we obtain the following geometric structure on $\Ind_{J(\Sigma)}^K\ \lambda_\Sigma$.

\begin{lemma}\label{lem:injective}
For any maximal compact subgroup $K$ of $G$ which contains $J(\Sigma)$, the set of representations $\Ind_{J(\Sigma_{x,\zeta})}^K\ \lambda_{\Sigma_{x,\zeta}}$ for which $G_x\subset K$ forms an injective system via the canonical inclusion maps
\[\Ind_{J(\Sigma_{x,\zeta})}^K\ \lambda_{\Sigma_{x,\zeta}}\hookrightarrow\Ind_{J(\Sigma_{x',\zeta'})}^K\ \lambda_{\Sigma_{x',\zeta'}},
\]
which exist whenever $(x,\zeta)\preceq(x',\zeta')$. In particular, one has the identity
\[\Ind_{J(\Sigma)}^K\ \lambda_\Sigma=\varinjlim\Ind_{J(\Sigma_{x,\zeta})}^K\ \lambda_{\Sigma_{x,\zeta}}.
\]
\end{lemma}

In fact, via this system we are able to exhaustively describe the isomorphism classes of irreducible components of $\Ind_{J(\Sigma)}^K\ \lambda_\Sigma$ (although not necessarily the multiplicity of each such isomorphism class).

\begin{proposition}\label{prop:irreducibles}
Suppose that $(x,\zeta)$ is maximal with respect to $\preceq$ and $G_x\subset K$. Then the representation $\lambda_{\Sigma_{x,\zeta}}$ of $J(\Sigma_{x,\zeta})$ is irreducible.

Moreover, every irreducible component of $\Ind_{J(\Sigma)}^K\ \lambda_\Sigma$ is isomorphic to a representation of the form $\Ind_{J(\Sigma_{x,\zeta})}\ \lambda_{\Sigma_{x,\zeta}}$ for some pair $(x,\zeta)$ which is maximal with respect to $\preceq$.
\end{proposition}

\begin{proof}
We freely make use of the results of \cite[\S 8]{KimYu2017} and \cite[\S 15]{Yu2001}. By carrying out the standard inductive argument from \emph{loc. cit.}, one finds that in order to show that $I_K(\lambda_{\Sigma_{x,\zeta}})$, the set of $k\in K$ which intertwine $\lambda_{\Sigma_{x,\zeta}}$ with itself, is equal to $J(\Sigma_{x,\zeta})$, from which the irreducibility of $\Ind_{J(\Sigma_{x,\zeta})}^K\ \lambda_{\Sigma_{x,\zeta}}$ follows, it suffices to show that $I_{K\cap G^0}(\zeta)=K\cap G_{[x]}^0$. Since $K$ is a maximal compact subgroup of $G$ which contains $G_x$, the group $K\cap G^0$ is a compact subgroup of $G^0$ which contains $G_x^0$. However, we know that $(x,\zeta)$ is maximal with respect to $\preceq$, which implies that $x$ is a vertex when viewed as a point in $\buil(G^0)$, i.e. that $G_x^0$ is a maximal compact subgroup of $G^0$, which implies that $K\cap G^0=G_x^0$. So, indeed, it must be the case that $I_{K\cap G^0}(\zeta)=K\cap G_{[x]}^0$, from which it follows that $\Ind_{J(\Sigma_{x,\zeta})}^K\ \lambda_{\Sigma_{x,\zeta}}$ is irreducible.

In order to see that every irreducible component of $\Ind_{J(\Sigma)}^K\ \lambda_\Sigma$ arises in such a way, it remains for us to observe that it follows from the construction of the various representations $\kappa_{\Sigma_{x,\zeta}}$ as Heisenberg--Weil lifts of $\theta_{\Sigma_{x,\zeta}}$ that $\Ind_{J(\Sigma)}^K\ \lambda_\Sigma$ decomposes as a sum of representations of the form $\Ind_{J(\Sigma_{x,\zeta})}^K\ \lambda_{\Sigma_{x,\zeta}}$, for $(x,\zeta)$ maximal. In particular, the set of representations of the form $\Ind_{J(\Sigma_{x,\zeta})}^K\ \lambda_{\Sigma_{x,\zeta}}$ exhausts the set of isomorphism classes of irreducible components of the representation
\[(\Ind_{J(\Sigma)}^K\ \sigma)\otimes(\Ind_{J(\Sigma)}^K\ \kappa_\Sigma);
\]
since there is a canonical injective map
\[\Ind_{J(\Sigma)}^K\ \lambda_\Sigma\rightarrow(\Ind_{J(\Sigma)}^K\ \sigma)\otimes(\Ind_{J(\Sigma)}^K\ \kappa_\Sigma),
\]
the claim follows.
\end{proof}

\bibliographystyle{amsalpha}
\addcontentsline{toc}{chapter}{Bibliography}
\bibliography{padicrefs}

\end{document}